\documentclass[11pt,epsf]{article}
\usepackage{graphicx}
\usepackage{amsthm}
\usepackage{amsfonts}
\usepackage{amssymb}
\title{A separation in modulus property of the 
zeros of a partial theta function}
\author{Vladimir Petrov Kostov\\ 
Universit\'e C\^ote d’Azur, CNRS, LJAD, France,\\  
e-mail: kostov@math.unice.fr} 
\date{}
\newtheorem{tm}{Theorem}
\newtheorem{conj}{Conjecture}
\newtheorem{defi}[tm]{Definition}
\newtheorem{rem}[tm]{Remark}
\newtheorem{rems}[tm]{Remarks}
\newtheorem{lm}[tm]{Lemma}
\newtheorem{ex}[tm]{Example}

\newtheorem{prop}[tm]{Proposition}
\newtheorem{nota}[tm]{Notation}
\begin{document} 
\maketitle 
\begin{abstract}
We consider the partial theta function 
$\theta (q,z):=\sum _{j=0}^{\infty}q^{j(j+1)/2}z^j$, 
where $z\in \mathbb{C}$ is a variable and $q\in \mathbb{C}$, $0<|q|<1$, 
is a parameter. Set $\alpha _0~:=~\sqrt{3}/2\pi ~=~0.2756644477\ldots$. 
We show that, for $n\geq 5$, for $|q|\leq 1-1/(\alpha _0n)$ and for 
$k\geq n$ there exists a unique 
zero $\xi _k$ of $\theta (q,.)$ satisfying the inequalities 
$|q|^{-k+1/2}<|\xi _k|<|q|^{-k-1/2}$; all these zeros are simple ones. 
The moduli of the remaining $n-1$ zeros 
are $\leq |q|^{-n+1/2}$. A {\em spectral value} of $q$ is a value for which 
$\theta (q,.)$ has a multiple zero. We prove the existence of the spectral 
values $0.4353184958\ldots \pm i\, 0.1230440086\ldots$ for which $\theta$ 
has double zeros $-5.963\ldots \pm i\, 6.104\ldots$.\\ 


\end{abstract}

\section{Introduction}

The series $\theta (q,z):=\sum _{j=0}^{\infty}q^{j(j+1)/2}z^j$ 
in the variables $q$ and $z$ is 
convergent for $q\in \mathbb{D}_1\backslash 0$, $z\in \mathbb{C}$ 
(here $\mathbb{D}_a$ 
denotes the open disk centered at the origin and of radius $a$). 
The series defines a {\em partial theta function}. The terminology is 
explained by the fact that 
the Jacobi theta function is 
the sum of the series $\Theta (q,z):=\sum _{j=-\infty}^{\infty}q^{j^2}z^j$ and 
the equality $\theta (q^2,z/q)=\sum _{j=0}^{\infty}q^{j^2}z^j$ 
holds true; ``partial'' means that in the case of $\theta$ the sum is taken 
only on $\mathbb{N}\cup 0$, not on $\mathbb{Z}$. 
For any fixed value of the variable $q$ (which we regard as a parameter), 
$\theta$ 
is an entire function in $z$.  

The most recent application of the function $\theta$ is connected with 
a problem about hyperbolic polynomials 
(i.e. real polynomials having all their zeros real). It has been discussed in 
the articles \cite{KaLoVi}, \cite{Ko2}, 
\cite{KoSh} and \cite{Ost}. These results are a continuation of an earlier 
study performed by Hardy, Petrovitch and Hutchinson 
(see \cite{Ha}, \cite{Hu} and \cite{Pe}). Other domains in which $\theta$ 
is used are statistical physics 
and combinatorics (see \cite{So}), 
asymptotic analysis (see \cite{BeKi}), the theory 
of (mock) modular forms (see \cite{BrFoRh}) and Ramanujan-type $q$-series 
(see \cite{Wa}). See more facts about $\theta$ in~\cite{AnBe} and~\cite{So}.

For $0<|q|\leq 0.108$ all zeros of $\theta (q,.)$ are distinct, see~\cite{Ko4}. 
In fact, a stronger statement holds true. We say that, for fixed $q$,  
the zeros of $\theta$ 
are {\em separated in modulus} if one can enumerate these zeros in such a way 
that their moduli form a strictly increasing sequence tending to infinity 
(which implies that all zeros are simple). The following lemma is close to 
results formulated independently by A. Sokal and J. Forsg{\aa}rd; 
in \cite{Ko6} it has been formulated in a weaker version claiming only the 
absence of multiple zeros although the proof is the same:

\begin{lm}\label{lm02}
For any $q\in \overline{\mathbb{D}_{c_0}}$, $c_0:=0.2078750206\ldots$, 
the zeros of the function $\theta$ 
are separated in modulus.
\end{lm}

\begin{nota}
{\rm For fixed $q$ we denote by $\mathcal{C}_k$, $k\in \mathbb{N}$, 
the circumference in 
the $z$-space $|z|=|q|^{-k-1/2}$. 
To denote the restriction to $\mathcal{C}_k$ of a given function 
in two variables we use the subscript $k$ (e.~g. $\theta _k$ stands 
for $\theta |_{\mathcal{C}_k}$).}
\end{nota}

\begin{proof}[Proof of Lemma~\ref{lm02}]
Consider for fixed $q$ the function 
$\theta$ restricted to each of the circumferences $\mathcal{C}_k$, 
$k\in \mathbb{N}$. Fix $k$. 
Then in the series of $\theta$ the term of largest modulus is 
$L:=z^kq^{k(k+1)/2}$ (one has $|L||_{|z|=|q|^{-k-1/2}}=|q|^{-k^2/2}$). 
The sum $M$ of the 
moduli of all other terms is smaller than  
$|q|^{-k^2/2}\tau (|q|)$, where $\tau :=2\sum _{\nu =1}^{\infty}|q|^{\nu ^2/2}$. 
Indeed, 

\begin{equation}\label{ML}
\begin{array}{rclcl}
M+|L|&=&\sum_{j=0}^{\infty}|q|^{j(j+1)/2-j(k+1/2)}&=&
|q|^{-k^2/2}\sum _{j=0}^{\infty}|q|^{(j-k)^2/2}\\ \\ 
&=&|q|^{-k^2/2}(1+2\sum _{\nu =1}^k|q|^{\nu ^2/2}+\sum _{\nu =k+1}^{\infty}|q|^{\nu ^2/2})
&<&|q|^{-k^2/2}(1+\tau (|q|))~.\end{array}
\end{equation}
The condition $1\geq \tau (|q|)$ is tantamount to 
$|q|\leq c_0$. Thus for $|q|\leq c_0$ one has $|L|>M$. 

One can also observe that the circumferences $|z|=|q|^{-k-1/2}$ separate the 
zeros of $\theta$ in the sense that no zero of $\theta$ lies on any of 
these circumferences for $|q|\leq c_0$. As we mentioned above, 
for $|q|\leq 0.108$ all zeros $\xi _k$ of $\theta$ are simple. 
For any $k$ fixed and for $|q|$ close to $0$ 
one has $\xi _k\sim -q^{-k}$ 
(see Proposition~10 in \cite{Ko2}). Hence for $k\in \mathbb{N}$ and 
$|q|\leq c_0$ one has 

\begin{equation}\label{cond1}
|q|^{-k+1/2}<|\xi _k|<|q|^{-k-1/2}~,
\end{equation} 
i.e. exactly one zero of $\theta$ lies between these two circumferences 
and all zeros are separated in modulus. One can continue analytically the 
zeros for $|q|\leq c_0$ and extend the inequalities $(\ref{cond1})$ 
to the domain 
$\mathbb{D}_{c_0}\backslash 0$. Thus the enumeration of the zeros of 
$\theta$ given by the 
increasing of the modulus is valid in $\mathbb{D}_{c_0}\backslash 0$.
\end{proof}

\begin{rem}
{\rm Using the same reasoning as the one in the proof of the lemma one can 
deduce that for $|q|\leq 0.2247945929\ldots$ 
the inequality $|\xi _1|<|q|^{-3/2}$  
holds true. To this end one has to 
consider instead of the condition $1\geq \tau (|q|)$ the inequality 
$1\geq 2\sum _{\nu =1}^k|q|^{\nu ^2/2}+\sum _{\nu =k+1}^{\infty}|q|^{\nu ^2/2}$ 
for $k=1$, see (\ref{ML}).} 
\end{rem} 

\begin{defi}
{\rm In what follows we say that, for a given $q$, {\em strong separation} 
of the zeros of $\theta$ 
occurs for $k\geq k_0$ 
in the sense that for any $k\geq k_0$ there exists a unique zero $\xi _k$ 
of $\theta$ which is simple and which satisfies condition $(\ref{cond1})$.}
\end{defi}  

For certain values of $q$ with $|q|>c_0$ (we call them {\em spectral values}) 
the function $\theta (q,.)$ has multiple zeros. 
It has been established in \cite{Ko5} 
that for any fixed value of the parameter $q$, the function 
$\theta$ has at most finitely-many multiple zeros. 
For $q\in (0,1)$ there exists a sequence of values of $q$, tending 
to $1$, for which 
$\theta (q,.)$ has double real negative zeros tending to $-e^{\pi}$, 
see \cite{Ko3}. When $q\in (-1,0)$, there exist two such sequences tending 
to $-1$ for which the corresponding double values of $\theta (q,.)$ tend 
to $\pm e^{\pi /2}$, see~\cite{Ko7}. The spectral number 
$\tilde{q}_1:=0.3092493386\ldots$ 
(which is the smallest positive one) is 
connected with hyperbolic polynomials that remain such when their highest 
degree monomial is deleted, see \cite{KaLoVi} and \cite{KoSh}. There is 
numerical evidence that there are infinitely-many complex (not real) spectral 
numbers as well.



We denote by $\alpha _0~:=~\sqrt{3}/2\pi ~=~0.2756644477\ldots$ the positive 
solution to the equation $-(2\pi ^2/3)\alpha +1/(2\alpha )=0$. 
In what follows we use the fact that for $n\geq 4$ 
one has $1-1/(\alpha _0n)>0$. The circumferences $\mathcal{C}_k$ being defined 
using half-integer exponents we replace in the formulation 
of the theorem below the condition $n\geq 4$ by the weaker 
condition $n\geq 5$.

\begin{tm}\label{tmsepar}
(1) For $n\geq 5$ and for $|q|\leq 1-1/(\alpha _0n)$ strong separation 
of the zeros $\xi _k$ of $\theta$ occurs for $k\geq n$.

(2) For these zeros $\xi _k$ one has $|\xi _k|\geq 336.2792102\ldots$.

(3) For any fixed $q$, all zeros of $\theta$ whose moduli are 
$\geq 4.685636519\ldots \times 10^5$, are strongly separated in modulus.

(4) For $0<|q|\leq 1/2$ strong separation of the zeros of $\theta$ occurs for 
$k\geq 4$.
\end{tm}

The paper is organized as follows. Section~\ref{sectionremarks} contains 
some remarks about the spectrum of $\theta$. Section~\ref{sectionproofs} 
contains the proof of Theorem~\ref{tmsepar}. Section~\ref{NPR} contains 
some notation used in Section~\ref{sectionclose0}. The latter contains the 
formulation of Proposition~\ref{3svs} claiming the existence of certain 
spectral values of $q$ in the disk $\mathbb{D}_{1/2}$. Proposition~\ref{3svs} 
is proved in Section~\ref{pr3svs} while Lemmas~\ref{lmsepar} and 
\ref{lmthetazz} used in its proof are proved in Section~\ref{pr2lm}. 

{\bf Acknowledgement.} The author acknowledges the kind hospitality 
of the University of Stockholm and of Brunel University during the visits 
to which he had fruitful discussions with 
B.Z.~Shapiro, J.~Forsg{\aa}rd and I.~Krasikov. Electronic discussions 
with A.~Sokal, A.E.~Eremenko and A.~Folsom were also very helpful.

\section{Some remarks about the zeros and the spectrum 
of $\theta$\protect\label{sectionremarks}}

~~~~(1) Part (3) of the theorem is an improvement 
of the basic result of \cite{Ko6}. 
The latter states that, for any $0<|q|<1$ and for $|z|\geq 8^{11}=8589934592$, 
all zeros of $\theta$ are simple. 

If $q$ is real, i. e. $q\in (-1,0)\cup (0,1)$, then all coefficients of 
$\theta (q,.)$ are real and a priori $\theta$ can have only real zeros 
and/or pairs of complex conjugate ones. Part (3) of the theorem implies that 
the moduli of the latter are  $\leq 4.685636519\ldots \times 10^5$. 
\vspace{1mm}

(2) For any $q\in (0,1)$, the function $\theta (q,.)$ has infinitely-many real 
negative zeros (and no positive ones), and the double zeros if any are the 
rightmost negative ones. They are local minima for $\theta$. In \cite{Ko3} 
it is proved that, for $q\in (0,1)$, the 
real positive spectral values 
of $\theta$ have the following asymptotic presentation: 
$\tilde{q}_s=1-(\pi /2s)+o(1/s)$, where $0<\tilde{q}_1<\tilde{q}_2<\cdots <1$ 
(a more precise presentation is obtained 
in \cite{Ko10}). For any $\gamma \in (0,1)$ one can enumerate all but 
finitely-many of the zeros of $\theta (q,.)$, $q\in (0, \gamma )$, so that:

\begin{enumerate}
\item 
For $0<q<\gamma <\tilde{q}_1$, $\theta (q,.)$ 
has all zeros real, negative, distinct and enumerated in 
the decreasing order. For any fixed index $k$, the corresponding zero $\xi _k$ 
is continuous in $q$.

\item 
When $q=\tilde{q}_s$, the zeros 
$\xi _{2s-1}<0$ and $\xi _{2s}<0$ coalesce and then give birth to a 
complex conjugate pair, see part (2) of Theorem~1 of~\cite{Ko2}. 
For any fixed index $k\geq 2s+1$, the zero $\xi _k$ is continuous in 
$q$ for $q\in (0, \tilde{q}_{s+1})$.
\end{enumerate}

Confluence of two real zeros of $\theta$ takes place also at negative 
spectral values; the asymptotics of the moduli of the spectral values for 
$q\in (-1,0)$ is of the form $1-(\pi /8s)+o(1/s)$, see~\cite{Ko7}.  
For $s$ odd (resp. for $s$ even) $\theta$ has negative double zeros 
which are its local minima (resp. positive double zeros which are its local 
maxima). For any $q\in (-1,0)$ the function $\theta (q,.)$ has 
infinitely-many positive and infinitely-many negative zeros. 
The negative double zeros are the rightmost negative real zeros and 
the positive double zeros are the second from the left positive real zeros. 
\vspace{1mm}

(3) It is shown in \cite{Ko8} that the zeros of $\theta$ 
are expanded in Laurent 
series in the parameter $q$ of the form 

\begin{equation}\label{A}
\xi _k=-1/q^k+(-1)^kq^{k(k-1)/2}(1+\Phi _k(q))~,
\end{equation} 
where $\Phi _k$ is a Taylor series 
with integer coefficients (in \cite{Ko8} the zeros are denoted by $-\xi _k$). 
Stabilization properties of the coefficients of the series $\Phi _k$ are 
proved in \cite{Ko8} and \cite{Ko9}.
\vspace{1mm}

(4) Part (1) of Theorem~\ref{tmsepar} and parts (2) and (3) of the present 
remarks imply that the radius of convergence of the series $(\ref{A})$ is 
$\leq 1-(\pi /k)+o(1/k)$. Indeed, this radius equals the distance from $0$ 
to the nearest singularity of the right-hand side of $(\ref{A})$. 
This singularity is not further from $0$ than $\tilde{q}_{[(k+1)/2]}$ (where 
$[]$ stands for the integer part of), 
see part (2) 
of the present remarks. Hence a priori the statement of part (1) of 
Theorem~\ref{tmsepar} could be improved only by looking for an 
inequality of the form $|q|\leq 1-1/(\alpha _1n)$ with a constant 
$\alpha _1\in (\alpha _0,1/\pi ]$, but not of the form 
$|q|\leq 1-\tau (n)$ with $\tau =o(1/n)$. 
\vspace{1mm}  

(5) Consider for $k\geq 5$ 
the function $q^k\xi _k=-1+(-1)^kq^{k(k+1)/2}(1+\Phi _k(q))$. 
It is holomorphic for $|q|<1-1/(\alpha _0k)$ and continuous for 
$|q|\leq 1-1/(\alpha _0k)$. Hence $|q^k\xi _k|\leq |q|^{-1/2}$, see 
(\ref{cond1}). The maximum of 
this modulus is attained for $|q|=1-1/(\alpha _0k)$, therefore 
$|q^k\xi _k|\leq (1-1/(\alpha _0k))^{-1/2}$. By the Cauchy inequalities, if one 
sets $q^k\xi _k=\sum _{j=0}^{\infty}h_{k,j}q^j$, then one obtains the estimation 
$|h_{k,j}|\leq (1-1/(\alpha _0k))^{-j-1/2}$. 

(6) It is shown in \cite{Ko5} that for any fixed $q\in \mathbb{D}_1$, 
and for $k$ sufficiently large, the function $\theta$ 
has a zero close to $-q^{-k}$;  
these are all but finitely-many of the zeros of $\theta$. 
This result is complementary to the ones 
of parts (1) and (4) of Theorem~\ref{tmsepar}. 

(7) The function $\theta$ has no zeros for $|z|\leq 1/2|q|$ 
(hence no zero for $|z|\leq 1/2$), see \cite{Ko2}. 
On the other hand, the radius of the disk in the $z$-space centered at $0$ 
in which $\theta$ has no zeros for any $q\in \mathbb{D}_1$ is not larger than 
$0.5616599824\ldots$. Indeed, consider the series 
$\theta ^1:=\theta (\omega ,z)$, $\omega :=e^{3i\pi /4}$, for $|z|<1$. 
The sequence $\{ \omega ^{j(j+1)/2}\}$ being $8$-periodic the sum $\theta ^1$ 
equals 

$$(\sum _{j=0}^7\omega ^{j(j+1)/2}z^j)/(1-z^8)~.$$
The zero of least modulus of its numerator is a simple one and equals  

$$z_0:=0.337553312314574\ldots +i\, 0.448909453205253\ldots ~,~~{\rm with}~~
|z_0|=0.5616599824\ldots ~.$$ 
Hence for $\rho \in (0,1)$ sufficiently close to $1$ 
the function $\theta (\rho e^{3i\pi /4},.)$ has a zero close to $z_0$. This 
follows from the uniform convergence as $\rho \rightarrow 1^-$ of 
$\theta (\rho e^{3i\pi /4},z)$ to $\theta (e^{3i\pi /4},z)$ on any compact 
subdomain of the unit disk in the $z$-space.

\section{Proofs\protect\label{sectionproofs}}

\begin{proof}[Proof of Theorem~\ref{tmsepar}]
It is well-known that all zeros of the Jacobi theta function $\Theta$ 
are simple 
(see \cite{Wi} and Chapter X of \cite{StSh}), 
so this is also the case of the function 
$\Theta ^*(q,z)=\Theta (\sqrt{q},\sqrt{q}z)=
\sum _{j=-\infty}^{\infty}q^{j(j+1)/2}z^j$. 
The following property is known as the 
{\em Jacobi triple product} (see \cite{Wi}):

$$\Theta (q,z^2)=\prod _{m=1}^{\infty}(1-q^{2m})(1+z^2q^{2m-1})
(1+z^{-2}q^{2m-1})~.$$
It implies the identity 

\begin{equation}\label{Theta*}
\Theta ^*(q,z)=\prod _{m=1}^{\infty}(1-q^m)(1+zq^m)(1+q^{m-1}/z)~.
\end{equation}
Clearly the zeros of $\Theta ^*(q,z)$ are all the numbers $\mu _s:=-1/q^s$, 
$s\in \mathbb{Z}$. In what follows we set

$$\begin{array}{cclcccl}
Q&:=&\prod _{m=1}^{\infty}(1-q^m)&~~,~~&
U&:=&\prod _{m=1}^{\infty}(1+zq^m)\\ \\ 
R&:=&
\prod _{m=1}^{\infty}(1+q^{m-1}/z)&~~,~~&G&:=&
\sum _{j=-\infty}^{-1}q^{j(j+1)/2}z^j=\Theta ^*-\theta \end{array}$$
Thus $\Theta ^*=QUR$. Obviously, for $|z|\geq 2$ and $|q|<1$ one has 

\begin{equation}\label{G}
|G|\leq \sum _{j=1}^{\infty}|q|^{j(j-1)/2}/2^j\leq \sum _{j=1}^{\infty}1/2^j=1~.
\end{equation}
The following lemma is part of Lemma~4 in \cite{Ko6}:

\begin{lm}\label{LLL}
Suppose that $|q|\leq 1-1/b$, $b>1$, and $|z|>1$. 
Then $|Q|\geq e^{(\pi ^2/6)(1-b)}$ and $|R|\geq (1-1/|z|)e^{(\pi ^2/6)(1-b)}$. 
\end{lm}
In particular, for $b=\alpha n$, $\alpha >0$, one obtains 

\begin{equation}\label{QR}
|Q|\geq e^{(\pi ^2/6)(1-\alpha n)}~~{\rm and}~~
|R|\geq (1-1/|z|)e^{(\pi ^2/6)(1-\alpha n)}~.
\end{equation}
If in addition $|z|\geq 2$, then $|R|\geq e^{(\pi ^2/6)(1-\alpha n)}/2$.

\begin{lm}\label{lmU}
Suppose that $|z|=|q|^{-n-1/2}$. Then 
$|U|\geq |q|^{-n^2/2}e^{-(\pi ^2/3)(\alpha n(\alpha n-1))^{1/2}}$.
\end{lm}

\begin{proof}
It follows from the definition of $U$ that 
\begin{equation}\label{U}
\begin{array}{ccl}
|U|&\geq&\prod _{m=1}^{n}(|q|^{-n-1/2+m}-1)
\prod_{m=1}^{\infty}(1-|q|^{m-1/2})\\ \\ &=&
|q|^{-n^2/2}\prod_{m=1}^{n}(1-|q|^{m-1/2})
\prod_{m=1}^{\infty}(1-|q|^{m-1/2})\\ \\&\geq &|q|^{-n^2/2}
(\prod_{m=1}^{\infty}(1-|q|^{m-1/2}))^2~.\end{array}
\end{equation}
Set $P:=\prod_{m=1}^{\infty}(1-|q|^{m-1/2})$, hence $|U|\geq |q|^{-n^2/2}P^2$. 
Taking logarithms one obtains 

$$\begin{array}{rcl}
\ln P&=&-\sum _{\nu =1}^{\infty}|q|^{\nu -1/2}
-(1/2)\sum _{\nu =1}^{\infty}|q|^{2\nu -1}-
(1/3)\sum _{\nu =1}^{\infty}|q|^{3\nu -3/2}-\cdots \\ \\ 
&=&-|q|^{1/2}/(1-|q|)-|q|/2(1-|q|^2)-|q|^{3/2}/3(1-|q|^3)-\cdots \\ \\ 
&=&(-|q|^{1/2}/(1-|q|))T~~,\\ \\ 
{\rm where}~~T&=&1+|q|^{1/2}/2(1+|q|)+|q|/3(1+|q|+|q|^2)+\cdots~.\end{array}$$
It is clear that $|q|^{s/2}/(s+1)(1+|q|+\cdots +|q|^s)<1/(s+1)^2$ (because 
$|q|^r+|q|^{s-r}\geq 2|q|^{s/2}$, $r=0$, $\ldots$, $[s/2]$). Recall that 
$\sum _{s=0}^{\infty}1/(s+1)^2=\pi ^2/6=1.6449\ldots$. Hence $T\in (0,\pi ^2/6)$ 
and $P\geq e^{-(\pi ^2/6)|q|^{1/2}/(1-|q|)}$. Fix $\alpha >0$. 
For $|q|\leq 1-1/\alpha n$ this implies 
$P\geq e^{-(\pi ^2/6)(\alpha n(\alpha n-1))^{1/2}}$ from which the lemma follows.
\end{proof} 

For 
$|q|\leq 1-1/\alpha n$ the minoration 
$|q|^{-n^2/2}e^{-(\pi ^2/3)(\alpha n(\alpha n-1))^{1/2}}$ of $|U|$ 
is not less than 
$(1-1/\alpha n)^{-n^2/2}e^{-(\pi ^2/3)(\alpha n(\alpha n-1))^{1/2}}$. The quantity 
$(1-1/\alpha n)^{-n}$ is decreasing as $n$ is increasing; it tends to 
$e^{1/\alpha}$ as $n$ tends to infinity. Therefore 
$(1-1/\alpha n)^{-n^2/2}\geq e^{n/(2\alpha )}$ and 
$|U|\geq e^{n/(2\alpha )}e^{-(\pi ^2/3)(\alpha n(\alpha n-1))^{1/2}}\geq 
e^{n/2\alpha -(\pi ^2/3)\alpha n}$. 
Thus one obtains the estimation (see formula (\ref{Theta*}), 
conditions (\ref{QR}), the line that 
follows them and conditions (\ref{U}))

$$\begin{array}{ccccl}
|\Theta ^*|&=&|Q||U||R|&\geq&e^{(\pi ^2/6)(1-\alpha n)}e^{n/2\alpha -(\pi ^2/3)\alpha n}
(e^{(\pi ^2/6)(1-\alpha n)}/2)\\ \\ &&&=&e^{(\pi ^2/3)+
(-(2\pi ^2/3)\alpha +1/(2\alpha ))n}/2~.
\end{array}$$
For $\alpha =\alpha _0$ one gets $|\Theta ^*|\geq e^{\pi ^2/3}/2$. 
Recall that $\theta =\Theta ^*-G$. For the restrictions $\theta _k$, 
$\Theta ^*_k$ and $G_k$ of these functions to the circumference 
$\mathcal{C}_k$ one has $|\Theta ^*_k|\geq e^{\pi ^2/3}/2$ and 
$|G_k|\leq 1$, therefore 
for $t\in [0,1]$ one has $|\Theta ^*_k-tG_k|\geq (e^{\pi ^2/3}/2)-1>0$. 
This means that 
no zero of the function $\Theta ^*_k-tG_k$ crosses the circumference 
$\mathcal{C}_k$ as $t$ increases from $0$ to $1$. This is true for any 
$k\geq n$. Thus to prove part (1) of the theorem there remains to lift the 
condition $|z|\geq 2$ which was used to obtain the estimation 
$|R|\geq e^{(\pi ^2/6)(1-\alpha n)}/2$. 

Suppose that $k\geq n\geq 5$ and $|q|\leq 1-1/(\alpha _0n)$. Then 

$$|\xi _k|\geq (1-1/(n\alpha _0))^{-n+1/2}\geq (1-1/(5\alpha _0))^{-9/2}=
336.2792102\ldots ~.$$ 
(We use the fact that the function $(1-1/(x\alpha _0))^{-x+1/2}$ is 
decreasing for $x\geq 5$.) This proves part (2) of the theorem and also 
lifts the restriction $|z|\geq 2$. Now part (1) of the theorem is completely 
proved.

Suppose that $1-1/(\alpha _0(n-1))<|q|\leq 1-1/(\alpha _0n)$. The zero 
$\xi _n$ of $\theta$ is the one of smallest modulus among its zeros 
strongly separated 
in modulus which are mentioned in parts (1) and (2) of the theorem. 
One has 

$$|\xi _n|\leq |q|^{-n-1/2}<(1-1/(\alpha _0(n-1)))^{-n-1/2}~.$$ 
The right-hand side is maximal for $n=5$; 
the corresponding value is $4.685636519\ldots \times 10^5$. 
This proves part (3) of the theorem.


To prove part (4) of the theorem we perform the same reasoning as above yet 
we use more accurate 
inequalities. In particular, we consider $|z|$ to be not less than 
$|q|^{-n-1/2}$ which for $n\geq 3$ and $|q|\leq 1/2$ implies 
$|z|\geq d_0:=11.31370850\ldots$. This allows to make the estimations 

$$\begin{array}{cclccclc}
|R|&\geq &\prod _{m=1}^{\infty}(1-2^{1-m}/d_0)&=:&r_0&=&
0.8333799934\ldots&,\\ \\ 
P&\geq &\prod_{m=1}^{\infty}(1-2^{1/2-m})&=:&p_0&=&
0.1298980722\ldots&{\rm and}\\ \\ 
|Q|&\geq &\prod _{m=1}^{\infty}(1-2^{-m})&=:&q_0&=&0.2887880952\ldots ~.
\end{array}$$
Recall that (see (\ref{U})) $|U|\geq |q|^{-n^2/2}\prod_{m=1}^{n}(1-|q|^{m-1/2})P$.
For $0<|q|\leq 1/2$ and $n\geq 3$ this product is minimal for $|q|=1/2$ and 
$n=3$ when it equals 

$$u_0:=2^{9/2}\times 0.1558689591\ldots \times p_0=0.4581390612\ldots ~.$$
Thus one obtains the estimation 

$$|\Theta ^*|\geq |R||Q||U|\geq r_0q_0u_0=0.1102604290\ldots ~.$$
On the other hand one has 

$$|G|\leq \sum _{j=1}^{\infty}d_0^{-j}(1/2)^{j(j-1)/2}=0.09213257671\ldots <
0.1102604290\ldots ~.$$
As in the proof of Theorem~\ref{tmsepar} one concludes that for 
$0<|q|\leq 1/2$ strong separation of the zeros 
of $\theta$ occurs for $k\geq 4$.

\end{proof}

\section{Notation and preliminary remarks\protect\label{NPR}}

In this section we fix some notation which is to be used in next section. 
We set $\rho :=0.4353184958$, $\tau :=0.1230440086$. Observe that 
these are rational numbers; for infinite 
decimal fractions we write 
$0.4353184958\ldots$ and $0.1230440086\ldots$. 
We set $\varepsilon :=2\times 10^{-10}$ and we denote by 
$U\subset \mathbb{C}$ the rectangle (in the $q$-space) 

$$\begin{array}{clcrcclcr}
U:=\{~{\rm Re}q&\in [~~~\rho -\varepsilon &,&\rho +\varepsilon ~~~]&,&
{\rm Im}q&\in  
[~~~\tau -\varepsilon &,&\tau +\varepsilon ~~~]~\} \\&= [0.4353184956&,&
0.4353184960]&&&= [0.1230440084&,&0.1230440088]
\end{array}$$ 
We define the rectangle $V\subset \mathbb{C}$ (in the $z$-space) as the set 

$$V:=\{ ~{\rm Re}z\in [-5.961,-5.965]~,~{\rm Im}z\in [6.102,6.106]~\} ~.$$
We denote its vertices as follows:

$$A=(-5.965,6.102)~~,~~B=(-5.961,6.102)~~,~~C=(-5.961,6.106)~~,~~
D=(-5.965,6.106)~.$$
By abuse of language we denote by $A$ also the complex number 
$-5.965+i\, 6.102$ etc. 

In what follows we consider several functions which are defined after the 
function $\theta$. The subscripts $z$ and $q$ mean partial derivations, 
e.~g. $\theta _z:=\partial \theta /\partial z$,
$\theta _{qz}:=\partial ^2\theta /\partial q\partial z$ etc. Thus 
\begin{equation}\label{thetaqz}
\theta _{qz}=1+6q^2z+18q^5z^2+40q^9z^3+
75q^{14}z^4+126q^{20}z^5+196q^{27}z^6+\cdots ~.
\end{equation} 

The subscript $(k)$ means truncation, i. e. 
$\theta _{(k)}:=\sum _{j=0}^kq^{j(j+1)/2}z^j$. 

We set $\theta ^*:=(1/2q^3)\theta _{zz}$. It is clear that 

\begin{equation}\label{thetazz}
\theta ^*=1+3q^3z+6q^7z^2+10q^{12}z^3+
15q^{18}z^4+21q^{25}z^5+28q^{33}z^6+\cdots ~.
\end{equation} 
 
For $5\leq j\leq n\leq \infty$ 
we set $r_j\in [0,1]$, $r:=(r_5,r_6,\ldots ,r_n)$ if $n<\infty$ or 
$r:=(r_5,r_6,\ldots )$ if not. We define the family of functions  

$$\theta _r(q,z):=1+qz+q^3z^2+q^6z^3+q^{10}z^4+
\sum _{j=5}^nr_jq^{j(j+1)/2}z^j~.$$

For a function $f(q,z)$ defined on $U\times V$ 
we denote by $DR[f]$ (resp. $DI[f]$) the maximal 
possible absolute value of the difference between the values of 
Re$f$ (resp. Im$f$) at two 
different points of $U\times V$. Obviously, for two functions $f(q,z)$ 
and $g(q,z)$ one has 

\begin{equation}\label{DR1ineq}
DR[f+g]\leq DR[f]+DR[g]~~{\rm and}~~DI[f+g]\leq DI[f]+DI[g]~.
\end{equation}

\section{The spectral values closest to $0$
\protect\label{sectionclose0}}

In the present section we consider the restriction of $\theta$ to the disk 
$\mathbb{D}_{1/2}$. We prove in the next section the following 

\begin{prop}\label{3svs}
For $0<|q|\leq 1/2$ the function $\theta$ has at least three spectral values 
which equal

$$\tilde{q}_1:=0.3092493386\ldots ~~{\rm and}~~
v_{\pm}:=0.4353184958\ldots \pm i\, 
0.1230440086\ldots ~.$$
The function $\theta (\tilde{q}_1,.)$ has a double real negative zero 
$-7.5032559833\ldots$, 
all its other zeros are real, negative, simple and $<-8$. The function 
$\theta (v_{\pm},.)$ has a simple zero $-3.277\ldots \mp i\, 1.483\ldots$, a 
double zero $-5.963\ldots \pm i\, 6.104\ldots$, its other zeros are simple 
and of modulus larger than $|v_{\pm}|^{-7/2}=16.06050040\ldots$, see 
part (4) of Theorem~\ref{tmsepar}.
\end{prop}
In the present section we present a hint why the following conjecture 
should be true:

\begin{conj}\label{conjclose0}
The spectral values $\tilde{q}_1$ and $v_{\pm}$ are the only 
spectral values of $\theta$ for $|q|\leq 1/2$. 
\end{conj}
(These spectral numbers are mentioned in the lectures of A.~Sokal.) 

\begin{proof}[Hint of a proof]
One can approximate $\theta$ by its truncations 
$\theta _{(s)}:=\sum _{j=0}^sq^{j(j+1)/2}z^j$. We consider $\theta _{(s)}$ 
as a degree 
$s$ polynomial in $z$. The values of $q$ for which the latter  
has multiple zeros are the values for which one has 
Res$(\theta _{(s)},\partial \theta _{(s)}/\partial z,z)=0$. 

For $|q|\leq 1/2$, the truncations with $s=13$, $\ldots$, $24$  
have multiple zeros 
for $q\approx \tilde{q}_1$, for $q\approx v_{\pm}$ and for no other 
value of $q$. Up to the $10$th 
decimal, these values of $q$ are the same for $s=13$, $\ldots$, $24$. 
This fact makes 
Conjecture~\ref{conjclose0} plausible, 
but does not provide a rigorous proof of it. 

It is proved in \cite{Ko11} that $\tilde{q}_1$ is 
the only spectral value of $\theta$ belonging to $\mathbb{D}_{0.31}$. 

Part (4) of Theorem~\ref{tmsepar} implies that for $0.31\leq |q|\leq 0.5$ 
the multiple zeros of $\theta$ (if any) have a modulus 
$\leq 0.31^{-3.5}=60.28844350\ldots$. 
The first term of the series $\theta -\theta _{(18)}$ equals $q^{190}z^{19}$. 
For $|q|=0.5$ and $|z|=0.31^{-3.5}$ its modulus equals 
$4.253517108\ldots \times 10^{-24}$. Therefore one expects that 
the truncation $\theta _{(18)}$ provides sufficient accuracy in the 
computation of the 
three spectral values of $\theta$ closest to $0$. 
\end{proof}

\begin{rems}
{\rm (1) The approximations up to the $6$th decimal of the 
first $25$ real positive spectral values are equal to (see \cite{KoSh})

$$\begin{array}{lllll}
0.309249~,~&0.516959~,~&0.630628~,~&0.701265~,~&0.749269~,\\ 
0.783984~,~&0.810251~,~&0.830816~,~&0.847353~,~&0.860942~,\\
0.872305~,~&0.881949~,~&0.890237~,~&0.897435~,~&0.903747~,\\ 
0.909325~,~&0.914291~,~&0.918741~,~&0.922751~,~&0.926384~,\\ 
0.929689~,~&0.932711~,~&0.935482~,~&0.938035~,~&0.940393~.\end{array}$$ 

(2) The approximations up to the $6$th decimal of the moduli 
of the first $8$ negative 
spectral values are equal to (see \cite{Ko7}) 

$$
0.727133~,~0.783742~,~0.841601~,~0.861257~,~ 
0.887952~,~0.897904~,~0.913191~,~0.919201~.
$$

(3)  One has 
$|v_{\pm}|=0.4523737623\ldots$. As we said above, 
the spectral value $\tilde{q}_1$ 
is the closest to $0$. Of the other 
real spectral values closest to the border of $\mathbb{D}_{1/2}$ 
(and also to $0$) is 
$w:=0.5169593598\ldots$. It seems that it is the next closest to $0$ 
(after $v_{\pm}$) among all spectral values because the next after $w$ 
closest to $0$ of the zeros of 
Res$(\theta _{(18)},\partial \theta _{(18)}/\partial z,z)$ 
equal $0.5373389195\ldots \pm i\, 0.1803273369\ldots$. 
Their modulus equals $0.5667901400\ldots$.}
\end{rems}

\section{Proof of Proposition~\protect\ref{3svs}
\protect\label{pr3svs}}

The statements of the proposition concerning $\tilde{q}_1$ and the double 
zero $-7.5\ldots$ are proved in \cite{KoSh} and \cite{Ko11}. The lemmas 
from this section are proved in the next one.

\begin{lm}\label{lmsepar}
For $q\in U$ and $|z|=|q|^{-2}$, and for any $r$ as in Section~\ref{NPR}  
one has  
$\theta _r(q,z)\neq 0$.
\end{lm}

\begin{rem}
{\rm  The family of functions $\theta _r$ 
contains, in particular, the functions $\theta _{(18)}$ and $\theta$. 
Lemma~\ref{lmsepar} implies that the smallest of the moduli of the 
zeros of any of the functions $\theta _r$ is less than $|q|^{-2}$ when 
$q\in U$. Indeed, the smallest modulus zero of 
$\theta _{(18)}(\rho +i\, \tau ,.)$ equals 
$-3.27794407050033\ldots -i\, 0.148307531004121\ldots$, its modulus is 
less than $4$ while for $q\in U$ one has $5<|q|^{-2}$; this can be checked 
numerically. There 
remains to apply a continuity argument.

By part (4) of Theorem~\ref{tmsepar}, for $q\in U$ the function $\theta$ 
has two simple zeros or one double zero whose moduli belong to the interval 
$[|q|^{-2},|q|^{-7/2}]$.}
\end{rem}

\begin{lm}\label{lmthetazz}
(1) For $(q,z)\in U\times V$ one has 
Re$(\theta ^*)\in (0.03,0.08)$ and 
Im$(\theta ^*)\in (0.15,0.20)$.

(2)  For $(q,z)\in U\times V$ one has 
Re$(\theta _{qz})\in (-0.70,0.84)$ and 
Im$(\theta _{qz})\in (-2.33,-0.79)$.

\end{lm}

For $q=\rho +i\, \tau$, the function $\theta _z(q,.)$ has a zero 
which equals $-5.963\ldots+i\, 6.104\ldots$. This is a simple zero of 
$\theta _z(q,.)$, see part (1) of Lemma~\ref{lmthetazz}. 
Hence it can be considered as a function $\eta (q)$ (as long as $q\in U$ 
and the values of this function belong to $V$). 

Consider the level sets $\{ \theta ={\rm const}\}$. 
The function $\theta$ satisfies the equality 

\begin{equation}\label{thetadiffeq}
2q\theta _q=z((z\theta )_{zz})=
z^2\theta _{zz}+2z\theta _z~.
\end{equation}  
As $\theta _z=0$ along the graph of $\eta$ and as 
$\theta _{zz}\neq 0$ in $U\times V$ 
(see part (1) of Lemma~\ref{lmthetazz}), one deduces from (\ref{thetadiffeq}) 
that $\theta _q\neq 0$ along the graph of $\eta$. Hence the 
level sets $\{ \theta ={\rm const}\}$ are locally analytic 
at their intersection points with this graph and their tangent 
spaces at these points are parallel to the $z$-space.

Differentiating the equality $\theta _z(q,\eta (q))=0$ w.r.t. $q$ one gets 
$\eta _q=-\theta _{qz}/\theta _{zz}=-\theta _{qz}/(2q^3)\theta ^*$. 
Lemma~\ref{lmthetazz} implies that 

\begin{equation}\label{etaeq}
|\eta _q|\leq (0.84^2+2.33^2)^{1/2}/2(q^*)^3(0.03^2+0.15^2)^{1/2}
\leq 87.44992430\ldots <87.45~,
\end{equation}
where $q^*:=(0.4353184956^2+0.1230440084^2)^{1/2}=0.4523737621\ldots$ 
is the smallest possible modulus of a number from $U$. 

$$\begin{array}{cccclcl}{\rm Set}~&q_a&:=&&
0.4353184958244864&+&i\, 0.1230440085519491~,\\& 
z_a&:=&-&5.963923719619588&+&i\, 6.104775174235743~.\end{array}$$ 
One can check numerically that:

\begin{equation}\label{approx}
\begin{array}{ccllllll}
\theta (q_a,z_a)&=
&-&1.6\ldots \times 10^{-15}&+&i\, 2.8\ldots \times 10^{-16}&=:&\chi _0~,\\ 
\theta _z(q_a,z_a)&=&-
&8.0\ldots \times 10^{-16}&+&i\, 0.0\ldots \times 10^{-15}&=:&\lambda ^*~.
\end{array}
\end{equation}
Consider for $q=q_a$ 
the vector field $\dot{z}=-1/\theta _{zz}$ 
(we denote the time by $\lambda$, i.~e. $\dot{z}=dz/d\lambda$). 
Its phase curve which for $\lambda =0$ passes through $z=z_a$, for  
$z=\lambda ^*$ passes through a point $z^*$ with $\theta _z(q_a,z^*)=0$. As 
$|\theta _{zz}|\geq (0.03^2+0.15^2)^{1/2}$ (see part (1) of 
Lemma~\ref{lmthetazz}), one has 

\begin{equation}\label{estimz}
|z^*-z_a|\leq |\lambda ^*|/(0.03^2+0.15^2)^{1/2}=5.2\ldots \times 10^{-15}
\end{equation}
The restriction $W$ of the subset $\{ \theta _z=0\}$ to the cartesian product 
$U\times V$ is locally a smooth complex curve. The set $W$ is the graph of 
a function, continuous on $U$ and analytic inside $U$. Indeed, consider 
inequalities (\ref{etaeq}). Denote by $(q',z')$ and $(q'',z'')$ any two 
points of $U\times V$. The distance between any two points of $U$ is 
$\leq 2\varepsilon$, therefore 
$|z'-z''|\leq (2\varepsilon )\times 87.45=3.498\times 10^{-8}$. Set 
$z':=z^*$. The last inequality, 
combined with inequality (\ref{estimz}) and the definition of $z_a$ and $V$, 
implies that $z''\in V$. Hence for any $q\in U$, the value of $\eta$ belongs 
to the set $V$. Analyticity and continuity of the function $\eta (q)$ 
follow from the fact that $\eta$ is a simple zero of $\theta _z$.  

Denote by $I\subset \mathbb{C}$ the segment with extremities at 
$0$ and $\lambda ^*$. The maximal possible value of $|\theta _z|$ 
at a point of the phase curve (with $\lambda \in I$) is not larger than 

$$\left( \max _{U\times V}|\theta _{zz}|\right) \times |\lambda ^*|\leq 
(0.08^2+0.20^2)^{1/2}\lambda ^*=1.7\ldots \times 10^{-16}=:\mu _0~.$$  
One has $\theta (q_a,z_a)=\chi _0$. 
Therefore 

\begin{equation}\label{GGG}
|\theta (q_a,z^*)|\leq |\chi _0|+|z^*-z_a|\mu _0< 
1.625\ldots \times 10^{-15}~.
\end{equation}
Define a vector field on $W$, with time $q$, by the formulas 
$q_q=1$, $z_q=-\theta _{qz}/\theta _{zz}=-\theta _{qz}/(2q/z^2)\theta _q$ 
(see (\ref{thetadiffeq})). Introduce as new time the value of $\theta$. Hence 
the vectorfield is defined by the formulas 

\begin{equation}\label{vf}
dq/d\theta =1/\theta _q~~,~~dz/d\theta =-\theta _{qz}/(2q/z^2)(\theta _q)^2~.
\end{equation}
Denote by $q^{\dagger}$ the value of $q$ (if it exists) 
for which the phase curve of this 
vector field with initial condition $(q,z)=(q_a,z^*)$ passes through 
a point $(q^{\dagger},\tilde{z})$ such that $\theta (q^{\dagger},\tilde{z})=0$. 
We want to show that $q^{\dagger}\in U$ (hence 
$(q^{\dagger},\tilde{z})\in U\times V$) which implies that this value of $q$ 
indeed exists.

For $(q,z)\in W$ one has $\theta _q=q^2z^2\theta ^*$ (this follows from 
(\ref{thetadiffeq}) and the definition of $\theta ^*$). The extremal values of 
$\arg (q^2z^2)$ (for $(q,z)\in U\times V$) are obtained for 

$$q=\rho +\varepsilon +i\, (\tau -\varepsilon ),~z=C~~{\rm and}~~
q=\rho -\varepsilon +i\, (\tau +\varepsilon ),~z=A~.$$
They equal respectively

$$-1.043893693643218\ldots ~~{\rm and}~~-1.042567942295371\ldots ~.$$
The extremal values of $|q^2z^2|$ are obtained for 

$$q=\rho -\varepsilon +i\, (\tau -\varepsilon ),~z=B~~{\rm and}~~
q=\rho +\varepsilon +i\, (\tau +\varepsilon ),~z=D~.$$
They are equal to

$$3.858934465358369\ldots ~~{\rm and}~~3.861493307333390\ldots ~.$$
By part (1) of Lemma~\ref{lmthetazz} one has

$$\begin{array}{rcclclc}
\arg (\theta ^*)&\in &[&\, 1.080839000541168\ldots &,&1.421906379185399\ldots &]\\ 
|\theta ^*|&\in &[&\, 0.1529705854077835\ldots &,&0.2154065922853802\ldots &]
\end{array}$$ 
Thus
$$\begin{array}{rcclclc}
\arg (q^2z^2\theta ^*)&\in &
[&\, 0.036945306897950\ldots &,&0.379338436890028\ldots &]\\
|q^2z^2\theta ^*|&\in &[&\, 0.5903034642161417\ldots &,&0.8317911144654879\ldots &]
\end{array}$$
and 
\begin{equation}\label{FFF}
\begin{array}{rcclclc}
\arg (1/\theta _q)&\in &
[&\, -0.379338436890028\ldots &,&-0.036945306897950\ldots &]\\  
|1/\theta _q|&\in &[&~~\, 1.202224912732572\ldots &,
& ~~1.694043929299806\ldots &]
\end{array}
\end{equation}
The quantity $q^{\dagger}$ is obtained by integrating the value of $1/\theta _q$ 
when the value of $\theta$ runs over the segment with extremities 
$\theta (q_a,z^*)$ and $0$. Inequalities (\ref{GGG}) and (\ref{FFF}) imply that 
$|q^{\dagger}-q_a|<10^{-10}$ hence $q^{\dagger}\in U$.

\section{Proofs of Lemmas~\protect\ref{lmsepar} and 
\protect\ref{lmthetazz}\protect\label{pr2lm}}

In the proofs of Lemmas~\ref{lmsepar} and \ref{lmthetazz} we use the 
following example:

\begin{ex}\label{ex1}
{\rm Consider the monomial 
$10q^{12}z^3$ (this is the fourth monomial of $\theta ^*$, see (\ref{thetazz})). 
Set 
$$\delta :=|2(1+i)\varepsilon /(\rho +i\, \tau -(1+i)\varepsilon )|=
1.250482394\ldots \times 10^{-9}~.$$ 
If instead of $q=\rho +i\tau$ one chooses another value of 
$q$ from the rectangle $U$, then:   

\begin{enumerate}
\item 
$|q|$ is multiplied by a real number from the interval 
$[1-\delta ,1+\delta ]$. 
Indeed, the numbers from $U$ have positive real and imaginary parts. 
The maximal and minimal possible ratios of moduli of numbers from $U$ 
are 

$$|(\rho +i\, \tau +(1+i)\varepsilon )/
(\rho +i\, \tau -(1+i)\varepsilon )|^{\pm 1}$$
from which the claim follows.

\item 
$|10q^{12}z^3|$ is multiplied by a number from the interval 

$$[(1-\delta )^{12},(1+\delta )^{12}]=
[0.9999999628\ldots ,1.000000036\ldots ]~;$$ 
this results directly from 1. (If we consider 
the monomial $15q^{18}z^4$, then such a change of $q$ would multiply 
$|15q^{18}z^4|$ by a number from the interval 
$[(1-\delta )^{18},(1+\delta )^{18}]$ etc.) 

\item 
The argument of $q$ can change by not more than 
$$\arg ((\rho -\varepsilon +i\, (\tau +\varepsilon ))/
(\rho +\varepsilon +i\, (\tau -\varepsilon )))=
1.091393649\ldots \times 10^{-9}~.$$
Clearly, such changes of $q$ can change the real or the imaginary part 
of the sum of the first six monomials of $\theta ^*$ (see (\ref{thetazz})) 
by less than $10^{-7}$. 

\item 
If one has to change the value of $z$ by any value from $V$, then one 
multiplies $|z|$ by a number from the interval 
$[|B/D|,|D/B|]=[0.9999922545\ldots ,1.000007306\ldots ]$ 
while its argument changes by 
not more than $\arg (A/C)=0.0006628745824\ldots $. 
This can change the sum of the 
first six monomials of $\theta ^*$ 
by a term whose modulus is less than $10^{-2}$. The same is true if one changes 
simultaneously $q$ and $z$.
\end{enumerate}}
\end{ex}

\begin{proof}[Proof of Lemma~\ref{lmsepar}]
For $q\in U$ and $|z|=|q|^{-2}$ one has 

\begin{equation}\label{theta3}
|\theta _{r}-\theta _{(4)}|\leq \sum _{j=5}^{n}|q|^{j(j+1)/2}|q|^{-2j}= 
\sum _{j=5}^{n}|q|^{j(j-3)/2}\leq \sum _{j=5}^{\infty}|q|^{j(j-3)/2}<0.02~. 
\end{equation} 
As for the polynomial $S:=\theta _{(4)}=1+qz+q^3z^2+q^6z^3+q^{10}z^4$, 
when one sets 
$z:=|q|^{-2}(\cos t+i\sin t)$, for $q=\rho +i\tau$ one gets 

$$\begin{array}{ccl}
{\rm Re}S&=&1+2.127219493\ldots \, \cos t-0.601264628\ldots \, \sin t\\&&+
1.497715597\ldots \, \cos (2t)
-1.625862822\ldots \, \sin (2t)-0.08191367962\ldots \, \cos (3t)\\&&-
0.9966394282\ldots \, \sin (3t)
-0.1895137736\ldots \, \cos (4t)-0.07721972768\ldots \, \sin (4t)\\ 
{\rm Im}S&=&2.127219493\ldots \, \sin t+0.6012646284\ldots \, \cos t+
1.497715597\ldots \, \sin (2t)\\&&+1.625862822\ldots \, \cos (2t)-
0.08191367962\ldots \, \sin (3t)+0.9966394282\ldots \, \cos (3t)\\&&
-0.1895137736\ldots \, \sin (4t)+0.07721972768\ldots \, \cos (4t)
\end{array}$$     

When one varies the values of $q$ while remaining in the rectangle $U$ one 
cannot change any of the coefficients of these trigonometric polynomials 
by more than $10^{-5}$. Indeed, 

\begin{enumerate}
\item 
The coefficient of 
$\cos (kt)$ or $\sin (kt)$ equals $\pm |q|^{-2k}$Re$(q^{k(k+1)/2})$ or 
$\pm |q|^{-2k}$Im$(q^{k(k+1)/2})$, $1\leq k\leq 4$. 

\item 
The possible variations of 
$q^{k(k+1)/2}$ (for $q\in U$) can be deduced from parts 1 and 3 of 
Example~\ref{ex1}. The moduli 
of these numbers are $<0.46$. 

\item 
The minimal and maximal 
possible values of $|q|^{-2k}$ for $q\in U$ are obtained for 
$q=\rho +i\, \tau \pm (1+i)\varepsilon$. For $k=4$ they equal 
$570.1914944\ldots$ and 
$570.1914999\ldots$, i.~e. the modulus of 
their difference is $<6\times 10^{-6}$. For $k=1,2$ and $3$ this modulus 
is even smaller.
\end{enumerate}    

As the sum of the moduli of all coefficients of Re$S$ and Im$S$  
is less than $50$ and one has $0\leq |\cos (kt)|, |\sin (kt)|\leq 1$, 
the values of Re$S$ and Im$S$ 
can vary by less than $10^{-3}$ when $q\in U$.

For $q=\rho +i\, \tau$ and for 
$t\in [0,4]$ (resp. $t\in [4,4.5]$ or $t\in [4.5,5.5]$ or 
$t\in [5.5,2\pi ]$) 
one has Im$S>0.04>0.02$ (resp. Re$S<-0.04$ or Im$S<-0.04$ or 
Re$S>0.04$) hence $|S|>0.04>0.02$.   
One can prove that the claimed inequalities hold true in the 
mentioned intervals by showing (say, using MAPLE) that the 
corresponding equalities have no solutions in these intervals 
and that for at least one point of the interval 
there is strict inequality. Hence for $q\in U$ and for $|z|=|q|^{-2}$ one has 
$|\theta _{(4)}|>0.03>0.02\geq |\theta _{r}-\theta _{(4)}|$. The lemma follows 
now from the Rouch\'e theorem.
\end{proof}

\begin{proof}[Proof of Lemma~\ref{lmthetazz}]
Recall that equality (\ref{thetazz}) holds true.  
Our aim is to estimate the real and imaginary parts of the monomials in 
its right-hand side for $(q,z)\in U\times V$. 
First of all we list the values of 
the first several monomials (excluding the constant term $1$) for $z=A$ or  
$z=C$ and for 
$q=\rho +i\tau$. The monomials and their values are:

$$3q^3z~,~6q^7z^2~,~10q^{12}z^3~,~15q^{18}z^4~,~21q^{25}z^5$$
$$\begin{array}{llll}
~~~~~~z=A&~~~&~~~~~~z=C&\\ 
-2.368899634\ldots &-i\, 0.06868680921\ldots &-2.368835235\ldots 
&-i\, 0.07025653317\ldots \\ 
+1.600516792\ldots &+i\, 0.554377995\ldots &+1.599755638\ldots &+
i\, 0.5564907703\ldots \\ 
-0.2788813462\ldots &-i\, 0.3612445530\ldots &-0.2781559523\ldots 
&-i\, 0.3617900215\ldots \\ 
-0.009845358519\ldots &+i\, 0.0490842341\ldots 
&-0.009975161466\ldots &+i\, 0.0490564368\ldots \\ 
+0.002251080781\ldots &-i\, 0.0005556207520\ldots 
&+0.002252822699\ldots 
&-i\, 0.0005481355646\ldots \end{array}$$
Observe that the sum of the real (resp. the imaginary) parts of these five 
monomials belongs to the interval $(0.05,0.06)$ (resp. $(0.17,0.18)$),   
for $z=A$ and for $z=C$. 

Next we use Example~\ref{ex1}. 
As $DR[3q^3z+\cdots +21q^{25}z^5]\leq DR[3q^3z]+\cdots +DR[21q^{25}z^5]$, 
one can write 

\begin{equation}\label{DR1}
DR_1:=DR[3q^3z+\cdots +21q^{25}z^5]\leq 10^{-2}~~{\rm and}~~
DI_1:=DI[3q^3z+\cdots +21q^{25}z^5]\leq 10^{-2}~.
\end{equation}
Starting with $21q^{25}z^5$ the moduli of the monomials decrease faster than 
a geometric progression with ratio $1/10$. Therefore the sum of all 
other monomials  
(i. e. $28q^{33}z^6+36q^{42}z^7+\cdots$) 
contributes to the real and the imaginary part of 
$\theta ^*$ by less than $10^{-4}$. Thus for 
$(q,z)\in U\times V$ one has $DR[\theta ^*]<0.02$ and 
$DI[\theta ^*]<0.02$, i. e. 
Re$(\theta ^*)\in (0.03,0.08)$ and  
Im$(\theta ^*)\in (0.15,0.20)$. 
This proves part (1) of the lemma. 

\begin{rem}
{\rm The small possible variation of $\arg z$ when $z\in V$ 
implies that for $q=\rho +i\tau$ the extremal possible values of 
the real and imaginary parts of the indicated monomials are attained for 
$z=A$ and $z=C$, the values of $z\in V$ with largest and smallest possible 
arguments.} 
\end{rem}
To prove part (2) using the same scheme of reasoning we use equality 
(\ref{thetaqz}). 
The first several monomials (excluding $1$) and their values for $z=A$ or  
$z=C$ and for 
$q=\rho +i\tau$ are equal to:

$$6q^2z~,~18q^5z^2~,~40q^9z^3~,~75q^{14}z^4~,~126q^{20}z^5~,~196q^{27}z^6$$
$$\begin{array}{llcll}
~~~~~~z=A&&~~~&~~~~~~z=C\\
-10.16093686\ldots &+i\, 2.556447275\ldots &~~~&
-10.16255051\ldots &+i\, 2.549691538\ldots \\ 
+24.24583379\ldots &-i\, 5.358044687\ldots &&+24.25254021\ldots &-i\, 
5.325813590\ldots \\ 
-19.64443131\ldots &-i\, 1.712625563\ldots &&-19.64053031\ldots &-i\, 
1.751646827\ldots \\ 
+4.696498002\ldots &+i\, 3.697047043\ldots &&+4.686533589\ldots &
+i\, 3.709371862\ldots \\ 
-0.03562046677\ldots &-i\, 0.7334753216\ldots &&
-0.03318797953\ldots &-i\, 0.7335609426\ldots \\
-0.03337082701\ldots &+i\, 0.01773395043\ldots &&
-0.03343954113\ldots &+i\, 0.01760026851\ldots 
\end{array}$$
The sums of the real (resp. imaginary) parts of these monomials and $1$ 
belong to 
the interval $(0.06,0.08)$ (resp. $(-1.57,-1.55)$). 
Starting with $196q^{27}z^6$, the moduli of the monomials 
decrease faster than a geometric progression with ratio $1/10$. 

We established in the proof of part (1) inequalities (\ref{DR1}) 
using inequality 
(\ref{DR1ineq}). 
Compare the monomials $3q^3z$, $\ldots$, $21q^{25}z^5$ 
with the monomials $18q^5z^2$, $\ldots$, 
$196q^{27}z^6$. The respective numerical coefficients increase less than 
$10$ times, the degree of $q$ is larger by $2$ 
and the one of $z$ is larger by $1$. 
The maximal possible value of $|q^2z|$ in $U\times V$ is $<1.8$. 
Therefore 

$$\begin{array}{llllll}DR_2&:=&DR[18q^5z^2+\cdots +196q^{27}z^6]&\leq& 
10\times 1.8\times (DR_1+DI_1)\leq 0.72&~~{\rm and}\\
DI_2&:=&DI[18q^5z^2+\cdots +196q^{27}z^6]&\leq& 
10\times 1.8\times (DR_1+DI_1)\leq 0.72~.&\end{array}$$
Consider $\theta ^0:=\theta _{qz}-(\theta _{qz})_{(6)}=
288q^{35}z^7+405q^{44}z^8+\cdots$.  
It is true that $DR[6q^2z]<10^{-2}$, $DI[6q^2z]<10^{-2}$, 
$DR[\theta ^0]\leq 10^{-3}$ and 
$DI[\theta ^0]\leq 10^{-3}$. Hence 
$DR[\theta _{qz}]<0.76$ and $DI[\theta _{qz}]<0.76$ 
from which part (2) of the lemma follows.
\end{proof}

\end{document}